\documentclass[12pt]{amsart}
\usepackage{graphicx,stmaryrd}
\usepackage[headings]{fullpage}
\usepackage{amssymb,epic,eepic,epsfig,amsbsy,amsmath,amscd,color}
\numberwithin{equation}{section}
                        \theoremstyle{plain}
\usepackage{mathrsfs}

\newcommand\no[1]{}

\newtheorem{theorem}{Theorem}[section]
\newtheorem{thm}{Theorem}
\newtheorem{lemma}[theorem]{Lemma}
\newtheorem{corollary}[theorem]{Corollary}
\newtheorem{proposition}[theorem]{Proposition}

\newtheorem{claim}[theorem]{Claim}

\theoremstyle{definition}
\newtheorem{remark}[theorem]{Remark}

\def\BC{\mathbb C}

\def\BZ{\mathbb Z}
\def\BR{\mathbb R}

\def\la{\langle}
\def\ra{\rangle}

\DeclareMathOperator{\tr}{\mathrm tr}

\def\be { \begin{equation} }
\def\ee { \end{equation} }

\numberwithin{equation}{section}

\def\tr{\text{tr}}
\def\SL{{\rm SL}_2(\BR)}

\def\uSL{\widetilde{\SL}}

\DeclareMathOperator{\Tr}{Tr}

\begin{document}
\allowdisplaybreaks
\baselineskip16pt
\title[LO surgeries of double twist knots II]{Left orderable surgeries  of double twist knots II}

\begin{abstract}
A slope $r$ is called a left orderable slope of a knot $K \subset S^3$ if the 3-manifold obtained by $r$-surgery along $K$ has left orderable fundamental group. Consider  two-bridge knots $C(2m, \pm 2n)$ and $C(2m+1, -2n)$ in the Conway notation, where $m \ge 1$ and $n \ge 2$ are integers. By using \textit{continuous} families of hyperbolic $\mathrm{SL}_2(\BR)$-representations of knot groups, it was shown in  \cite{HT-genus1, Tr} that any slope in $(-4n, 4m)$ (resp. $[0, \max\{4m, 4n\})$) is a left orderable slope of $C(2m, 2n)$ (resp. $C(2m, - 2n)$) and in \cite{Ga} that any slope in $(-4n,0]$ is a left orderable slope of $C(2m+1,-2n)$. However, the proofs of these results are incomplete since the \textit{continuity} of the families of representations was not proved. In this paper, we complete these proofs and moreover we show that any slope in $(-4n, 4m)$ is a left orderable slope of $C(2m+1,-2n)$ detected by hyperbolic $\mathrm{SL}_2(\BR)$-representations of the knot group. 
\end{abstract}

\author{Vu The Khoi}
\address{Institute of Mathematics, VAST, 18 Hoang Quoc Viet, 10307, Hanoi, Vietnam}
\email{vtkhoi@math.ac.vn}
\thanks{The first author has been partially supported by the Vietnam National Foundation for Science and Technology Development (NAFOSTED) under grant number 101.04-2015.20.}

\author[M. Teragaito]{Masakazu Teragaito}
\address{Department of Mathematics and Mathematics Education, Hiroshima University, 
1-1-1 Kagamiyama, Higashi-Hiroshima, 739--8524, Japan}
\email{teragai@hiroshima-u.ac.jp}
\thanks{The second  author has been partially supported by JSPS KAKENHI Grant Number JP16K05149.}

\author{Anh T. Tran}
\address{Department of Mathematical Sciences,
The University of Texas at Dallas,
Richardson, TX 75080-3021, USA}
\email{att140830@utdallas.edu}
\thanks{
The third author has been partially supported by a grant from the Simons Foundation (\#354595).
}


\maketitle

\section{Introduction}

The study of left orderability of fundamental groups of $3$-manifolds obtained by Dehn surgeries along knots is motivated by the L-space conjecture of Boyer, Gordon and Watson \cite{BGW} which states that an irreducible rational homology 3-sphere is an L-space if and only if its fundamental group is not left orderable. Here a rational homology 3-sphere $Y$ is an L-space if its Heegaard Floer homology $\widehat{\mathrm{HF}}(Y)$ has rank equal to the order of $H_1(Y; \BZ)$, and a non-trivial group $G$ is left orderable if it admits a total ordering $<$ such that $g<h$ implies $fg < fh$ for all elements $f,g,h$ in $G$.

Many hyperbolic 3-manifolds are obtained by Dehn surgeries along knots. A slope $r$ is called a left orderable slope of a knot $K \subset S^3$ if the 3-manifold obtained by $r$-surgery along $K$ has left orderable fundamental group. Consider  two-bridge knots $C(2m, \pm 2n)$ and $C(2m+1, -2n)$ in the Conway notation, where $m \ge 1$ and $n \ge 2$ are integers. By using \textit{continuous} families of hyperbolic $\mathrm{SL}_2(\BR)$-representations of knot groups, it was shown in  \cite{HT-genus1, Tr} that any slope in $(-4n, 4m)$ (resp. $[0, \max\{4m, 4n\})$) is a left orderable slope of $C(2m, 2n)$ (resp. $C(2m, - 2n)$) and in \cite{Ga} that any slope in $(-4n,0]$ is a left orderable slope of $C(2m+1,-2n)$. However, the proofs of these results are incomplete since the \textit{continuity} of the families of representations was not proved. More precisely, \cite[Proposition 4.2]{HT-genus1}, \cite[Lemma 2.1]{Tr} and \cite[Proposition 4.2]{Ga} proved the existence of families of $\mathrm{SL}_2(\BR)$-representations of the knot groups but did not prove the continuity of these families. In this paper, we complete these proofs in Proposition \ref{solution-even1}, Proposition \ref{solution-even2} and Proposition \ref{solution-odd} respectively. Moreover,  we extend the range of left orderable slopes of $C(2m+1,-2n)$ detected by hyperbolic $\mathrm{SL}_2(\BR)$-representations of their knot groups. 

\begin{figure}[th] \label{fig}
\centerline{\psfig{file=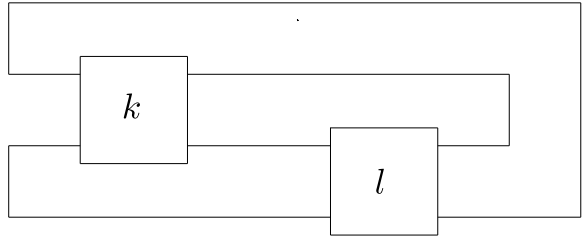,width=3in}}
\caption{The two-bridge knot/link $C(k,l)$ in the Conway notation.}
\end{figure} 

 \begin{thm} \label{main} 
 Let $m \ge 1$ and $n \ge 2$ be integers. Then any slope in $(-4n, 4m)$ is a left orderable slope of $C(2m+1,-2n)$ detected by hyperbolic $\mathrm{SL}_2(\BR)$-representations of the knot group. 
 \end{thm}
 
\begin{remark}
In \cite{Tr-I}, by following the method of Culler-Dunfield's paper \cite{CD} the third author used \textit{continuous} families of elliptic $\mathrm{SL}_2(\BR)$-representations of knot groups to show that if $K$ is a two-bridge knot of the form $C(2m,-2n)$,  $C(2m+1, 2n)$ or $C(2m+1, -2n)$, where $m \ge 1$ and $n \ge 1$ are integers, and 
\[
\text{LO}_K = \begin{cases} (-\infty,1) &\mbox{if } K = C(2m,-2n), \\ 
(-\infty, 2n-1) & \mbox{if } K = C(2m+1,2n),  \\ 
(3-2n, \infty) & \mbox{if } K = C(2m+1,-2n) \text{ and } n \ge 2, \end{cases}
\]
then any slope in $\text{LO}_K$ is a left orderable slope of $K$. 

Gao \cite{Ga} independently showed that if $K$ is a two-bridge knot of the form $C(2m+1, 2n)$ or $C(2m+1, -2n)$, where $m \ge 1$ and $n \ge 1$ are integers, and 
\[
\text{LO}'_K = \begin{cases}  
(-\infty, 1) & \mbox{if } K = C(2m+1,2n),  \\ 
(-1, \infty) & \mbox{if } K = C(2m+1,-2n) \text{ and } n \ge 2, \end{cases}
\]
then any slope in $\text{LO}'_K$ is a left orderable slope of $K$. Her proof also used families of elliptic $\mathrm{SL}_2(\BR)$-representations of knot groups. Note that $\text{LO}'_K$ is a subset of $\text{LO}_K$. 
\end{remark}

For more on the study of left orderable slopes of knots using hyperbolic $\mathrm{SL}_2(\BR)$-representations of knot groups, see 
\cite{BGW, HT-52, Ga-holonomy}.

This paper is organized as follows. In Section 2, we recall some facts about  the universal covering group $\uSL$ of  $\SL$ and we study the lifting problem of a connected curve of hyperbolic $\SL$-representations of knot groups. 
In Section 3, we review the Riley polynomial of two-bridge knots $C(k,-2p)$, whose zero locus describes all non-abelian representations of the knot group into $\mathrm{SL}_2(\BC)$. In Section 3 we prove the existence of continuous families of hyperbolic $\mathrm{SL}_2(\BR)$-representations of the knot groups of $C(2m, \pm 2n)$ and therefore fix the gaps in \cite{HT-genus1, Tr}. In Section 4 we prove the existence of continuous families of hyperbolic $\mathrm{SL}_2(\BR)$-representations of the knot groups of $C(2m+1, -2n)$ and use it to give a proof of Theorem \ref{main}.


\section{Lifting of a curve of hyperbolic representations}	
\subsection{The group $\uSL$} 
We recall some facts about the universal covering group $\uSL$ (see \cite{vu} pages 763-764). Let $\psi\colon \uSL \rightarrow \SL $ be the covering map. We can parameterize the universal covering group as
  $$\uSL= \left\{(\gamma,\omega)\mid |\gamma|<1, -\infty<w<\infty\right\}.$$ For an element $g=(\gamma,\omega)\in \uSL,$ we will write $g[1]=\gamma$ and $g[2]=\omega.$ 
  
  An element of $\uSL$ is called \emph{elliptic/parabolic/hyperbolic} if it covers a matrix in $\SL$ of the corresponding type. 

     The multiplication rule in the group $\uSL$ is given by
  $(\gamma, \omega)(\gamma', \omega')=(\gamma'', \omega'')$ where
  \begin{eqnarray*}\label{rule1}
  \gamma''&=&\frac{\gamma+\gamma'e^{-2i\omega}}{1+\bar
  	\gamma\gamma'e^{-2i\omega}} \\\label{rule2}
  \omega''&=&\omega+\omega'+ \arg(1+\bar
  \gamma\gamma'e^{-2i\omega}). 
  \end{eqnarray*}          
  
 Let $A=\begin{pmatrix} a & b\\ c & d    \end{pmatrix}$ be a matrix in $\SL$ then 
\begin{equation*}
\psi^{-1}(A)=\left \{ \left ( \frac{a-d+(b+c)i}{a+d+(b-c)i}, \arg(a+d+(b-c)i) +2n\pi \right ) \mid  n\in \BZ\right \}.
\end{equation*}
Here, the function argument takes value in the interval $(-\pi,\pi ].$  We note that if $\Tr(A)= a+d>0$ then in the above formula we have $\arg(a+d+(b-c)i)=\arctan(\frac{b-c}{a+d})\in (-\frac{\pi}{2},\frac{\pi}{2}).$    
 
\subsection{Lifting of a curve of hyperbolic representations}

For a knot $K$ in $S^3$, let $X$ be an open tubular neighborhood of $K$ and let $G(K) = \pi_1(X)$ be the knot group of $K$ which is the fundamental group of $X$. 
Let $\mu$ be a meridian and $\lambda$ the canonical longitude. 
 Recall that any representation $\rho\colon G(K) \to \SL$ can be lifted to a representation $\widetilde{\rho}$ into $\uSL$ because $H^2(K,\BZ)=0.$   We know from \cite{vu} that the lifts of $\rho$ come in a family. If we fix a lift $\widetilde{\rho}_0,$  then we have a $\BZ$-family of lifts $\widetilde{\rho}_n$ given by  $\widetilde{\rho}_n (g)=\widetilde{\rho}_0(g)h_n(g),$ where $h_n$ is the representation  
  \begin{eqnarray*}
  	G(K)  \longrightarrow& H_1(X)\equiv \BZ&\longrightarrow \uSL \\
  	\mu   \quad\longmapsto & 1  &\longmapsto (0,n\pi). 
  \end{eqnarray*} 
From this, we see that $\widetilde{\rho}_n (\lambda)$ does not depend on $n.$

An $\SL$ representation $\rho$ of a knot group $G(K)$ is called \emph{hyperbolic\/} if  $\rho(\mu)$ and  $\rho(\lambda)$ are hyperbolic elements.             
 
 Recall that a hyperbolic element $g\in \uSL$ can be conjugated to a unique normal form $(\tanh(a), k\pi)$ if and only if $k\pi -\frac{\pi}{2}<g[2]< k\pi +\frac{\pi}{2}.$ 
  Let us fix an arbitrary lift $\widetilde{\rho}_0\colon G(K)\rightarrow \uSL$ and suppose that $\widetilde{\rho}_0(\lambda)$ is conjugate to $(\tanh(b), k\pi).$ As noted above, the number $k$ does not depend on the chosen lift $\widetilde{\rho}_0.$  
  We call $k$ the \emph{index} of the representation $\rho.$
   
 \begin{lemma}\label{l1} Let $C$ be a connected curve of hyperbolic representations of  a knot group $G(K)$ into $\SL$. Then the indexes of all representations of $C$ are the same. \end{lemma}
 
 \begin{proof} 
 Note that as $\rho$ varies in $C,$ the image $\widetilde{\rho}_0(\lambda)$ varies in a connected component of hyperbolic elements. Since each connected component corresponds to a single value of $k$ (see Figure 1 of \cite{vu}),  the number $k$ is the same for all representations in $C.$  
\end{proof}

We will also say that the curve $C$ in the above lemma has \emph{index} $k.$

\begin{corollary} \label{c1}
Let $C$ be a connected curve of hyperbolic representations of  a knot group $G(K)$ into $\SL.$ If $C$ contains a reducible representation then it has index $0.$  
\end{corollary}

\begin{proof} 
Suppose that $\rho_1\in C $ is a reducible hyperbolic representation given by $$\rho_1(\mu_i)=  \begin{pmatrix} s & a_i\\ 0 & s^{-1}    \end{pmatrix},$$ where $\mu_i$ are generators of $G(K)$ which are conjugate to the standard meridian. Then we can connect $\rho_1$ to an abelian representation by using the curve $$\rho_t(\mu_i)=  \begin{pmatrix} s & ta_i\\ 0 & s^{-1}    \end{pmatrix}, t\in \BR.$$  It is easy to verify that if $\rho_1$ is a representation then so is $\rho_t$ for all $t.$ As the abelian representation $\rho_0$ has index 0, the corollary follows from the previous lemma. 
\end{proof}

The next proposition tells us how to find left orderable slopes of a knot, given a connected curve of hyperbolic representations of the knot group into $\SL$.

\begin{proposition} \label{p1}
\begin{itemize}
\item[\textup{(1)}]  Let $\{\rho_y\colon G(K)\rightarrow \SL\}_y$ be a connected curve of hyperbolic representations of index $0$, and choose a lift $\widetilde{\rho}_y\colon G(K)\rightarrow \uSL$ such that  $\widetilde{\rho}_y(\mu)=(\tanh a(y),0)$ and 
$\widetilde{\rho}_y(\lambda)=(\tanh b(y),0).$  If $p/q$ is a slope such that $p/q = -b(y)/a(y)$ for some $y$, then  $p/q$ is a left orderable slope of $K$.  
\item[\textup{(2)}]
Let $\{\rho_y\colon G(K)\rightarrow \SL\}_y$ be a connected curve of hyperbolic representations of index $k\neq 0$, and choose a lift $\widetilde{\rho}_y\colon G(K)\rightarrow \uSL$ such that   $\widetilde{\rho}_y(\mu)=(\tanh a(y),0)$ and 
$\widetilde{\rho}_y(\lambda)=(\tanh b(y),k\pi).$ If $p/q$ is a slope such that $p/q = -b(y)/a(y)$ for some $y$ and $p| k$, then  $p/q$ is a left orderable slope of $K$.  
\end{itemize}
\end{proposition}

\begin{proof} Let $X_{p/q}$ denote by 3-manifold obtained from $S^3$ by $p/q$-surgery along $K$.

(1) Since $\widetilde{\rho}_y(\mu^p\lambda^q)=(\tanh(pa(y)+qb(y)),0)=(0,0),$ $\widetilde{\rho}_y$ gives a representation from $\pi_1(X_{p/q})$ to $\uSL$. Note that $X_{p/q}$ is an irreducible 3-manifold (by \cite{HTh}) and $\widetilde{\mathrm{SL}_2(\BR)}$ is a left orderable group (by \cite{Be}). Hence, by Theorem 1.1 of \cite{BRW}, $\pi_1(X_{p/q})$ is a left orderable group. This means that $p/q$ is a left orderable slope of $K$.  
	
(2) We choose another lift $\widetilde{\rho_y}': G(K)\rightarrow \uSL$ such that $\widetilde{\rho_y}'(\mu)=(\tanh a(y),-\frac{kq}{p}\pi)$ and 
$\widetilde{\rho_y}'(\lambda)=(\tanh b(y),k\pi).$ 
 We then have $$\widetilde{\rho_y}'(\mu^p\lambda^q)=(\tanh(pa(y)+qb(y)),-kq\pi+kq\pi)=(0,0).$$
 Therefore $\widetilde{\rho_y}'$ gives a representation from $\pi_1(X_{p/q})$ to $\uSL$ and, as in (1), the assertion follows. 
\end{proof}

 \section{Representations of double twist knots}

Consider the two-bridge knot/link $C(k, l)$ in the Conway notation, where $k, l$ are integers such that $|kl| \ge 3$. Note that $C(k, l)$ is the rational knot/link corresponding to continued fraction $k+1/l$. It is easy to see that $C(k,l)$ is the mirror image of $C(l,k)=C(-k,-l)$. Moreover, $C(k, l)$ is a knot if $kl$ is even and is a two-component link if $kl$ is odd. In this paper, we only consider knots and so we can assume that $k>0$ and $l=-2p$ is even. 

Note that $C(k, -2p)$ is the mirror image of the double twist knot $J(k,2p)$ in \cite{HS}.  Then, by \cite{HS}, the knot group of $C(k, -2p)$ has a presentation 
\[
G(C(k, -2p)) = \la a, b \mid a w^{p}  =  w^{p} b\ra
\]
where $a, b$ are meridians and 
\[
w = \begin{cases} (ab^{-1})^m(a^{-1}b)^m &\mbox{if } k=2m, \\ 
(ab^{-1})^m ab (a^{-1}b)^m & \mbox{if } k=2m+1. \end{cases}
\]
Moreover, the canonical longitude of $C(k, -2p)$ corresponding to the meridian $\mu=a$ is $\lambda = (w^{p} (w^{p})^*a^{-2\varepsilon})^{-1}$, where $\varepsilon =0$ if $k=2m$ and $\varepsilon =2p$ if $k=2m+1$. Here, for a word $u$ in the letters $a,b$ we let $u^*$ be the word obtained by reading $u$ backwards. 

Suppose $\rho\colon G(C(k, -2p)) \to \mathrm{SL}_2(\BC)$ is a nonabelian representation. 
Up to conjugation, we may assume that 
\begin{equation} \label{repn}
\rho(a) = \left[ \begin{array}{cc}
M& 1 \\
0 & M^{-1} \end{array} \right] \quad \text{and} \quad 
\rho(b) = \left[ \begin{array}{cc}
M & 0 \\
2 - y & M^{-1} \end{array} \right]
\end{equation}
where $(M,y) \in \BC^2$ satisfies the matrix equation $\rho(a w^p) = \rho(w^p b)$. It is known that this matrix equation is equivalent to a single polynomial equation $R_{C(k, -2p)}(x,y) =0$, where $x= (\tr \rho(a))^2$ and $R_K(x,y)$ is the Riley polynomial of a two-bridge knot $K$, see \cite{Ri}. This polynomial can be described via the Chebychev polynomials as follows. 

Let $\{S_j(v)\}_{j \in \BZ}$ be the Chebychev polynomials in the variable $v$ defined by $S_0(v) = 1$, $S_1(v) = v$ and $S_j(v) = v S_{j-1}(v) - S_{j-2}(v)$ for all integers $j$. Note that $S_{j}(v) = - S_{-j-2}(v)$ and $S_j(\pm2) = (\pm 1)^j (j+1)$. Moreover, we see
that $S_j(v) = (s^{j+1} - s^{-(j+1)})/(s - s^{-1})$ for $v = s + s^{-1} \not= \pm 2$
from the recurrence relation. Using these identities one can prove the following.

\begin{lemma} \label{chev}
We have
\begin{enumerate}
\item   $S^2_j(v) - v S_j(v)S_{j-1}(v) + S^2_{j-1}(v) =1$ for any integer $j$,
\item $S_n(v)=\prod_{j=1}^n (v-2\cos \frac{j\pi}{n+1})$ for any positive integer $n$, 
\item $S_n(v)-S_{n-1}(v)=\prod_{j=1}^n (v-2\cos \frac{(2j-1)\pi}{2n+1})$ for any positive integer $n$.
\end{enumerate}
\end{lemma}

\begin{proof}
(1) By the recurrence relation,
\begin{align*}
S_j^2(v)-vS_j(v)S_{j-1}+S_{j-1}^2(v)&= S_j(v)(S_j(v)-vS_{j-1}(v))+S_{j-1}^2(v)\\
&= (vS_{j-1}(v)-S_{j-2}(v))(-S_{j-2}(v))+S_{j-1}^2(v)\\
&=S_{j-1}^2(v)-vS_{j-1}(v)S_{j-2}(v)+S_{j-2}^2(v).
\end{align*}
Since $S_1(v)^2-vS_1(v)S_0(v)+S_0^2(v)=1$, we have the conclusion.

(2)
For any positive integer $n$, $S_n(v)$ is a polynomial of degree $n$.
Since $\pm 2$ are not roots of $S_n(v)$,
all roots come from solving  $s^{n+1}-s^{-(n+1)}=0$, where $v = s+s^{-1}$.
Hence, the conclusion follows from the observation that $2\cos j\pi/(n+1)\ (j=1,2,\dots,n)$
give all roots of $S_n(v)$.

(3)
For any positive integer $n$, $S_n(v)-S_{n-1}(v)$ is a polynomial of degree $n$, and
\begin{align*}
S_n(v)-S_{n-1}(v)&=\frac{s^{n+1}-s^{-(n+1)} }{s-s^{-1}}-\frac{s^{n}-s^{-n}}{s-s^{-1}}\\
&=\frac{s^{n+1}-s^{-(n+1)}-s^n+s^{-n}}{s-s^{-1}}\\
&= \frac{s^{-(n+1)}}{s-s^{-1}}\cdot (s^{2n+2}-1-s^{2n+1}+s)\\
&=\frac{s^{-(n+1)}}{s-s^{-1}}(s-1)(s^{2n+1}+1).
\end{align*}
Hence all roots of $S_n(v)-S_{n-1}(v)$ come from solving $s^{2n+1}+1=0$.
That is, $2 \cos (2j-1)\pi/(2n+1)\ (j=1,2,\dots,n)$ give all the roots.
\end{proof}

The Riley polynomial of $C(k, -2p)$, whose zero locus describes all non-abelian representations of the knot group of $C(k, -2p)$ into $\mathrm{SL}_2(\BC)$, is
\[
R_{C(k, -2p)}(x,y)=S_{p}(t) - z S_{p-1}(t)
\]
where 
\begin{eqnarray*}
t &=& \tr \rho(w) = \begin{cases} 2+ (y+2-x)(y-2) S^2_{m-1}(y) &\mbox{if } k=2m, \\ 
2- (y+2-x)(S_m(y)-S_{m-1}(y))^2 & \mbox{if } k=2m+1, \end{cases}
\end{eqnarray*}
and
\begin{eqnarray*}
z &=& \begin{cases} 1 + (y+2-x) S_{m-1}(y) (S_{m}(y) - S_{m-1}(y)) &\mbox{if } k=2m, \\ 
1 - (y+2-x) S_{m}(y) (S_{m}(y) - S_{m-1}(y))  & \mbox{if } k=2m+1. \end{cases}
\end{eqnarray*}
Moreover, for the representation $\rho\colon G(C(k, -2p)) \to \mathrm{SL}_2(\BC)$ of the form \eqref{repn} the image of the canonical longitude $\lambda =  (w^{p} (w^{p})^*a^{-2\varepsilon})^{-1}$ has the form 
$\rho(\lambda) = \left[ \begin{array}{cc}
L & * \\
0 & L^{-1} \end{array} \right]$, where 
\[
L = - \frac{M^{-1}(S_m(y) - S_{m-1}(y)) - M(S_{m-1}(y) - S_{m-2}(y))} {M(S_m(y) - S_{m-1}(y)) - M^{-1} (S_{m-1}(y) - S_{m-2}(y))} \quad  \text{if~}  k=2m
\]
and  
\[
L = - M^{4p} \frac{M^{-1}S_m(y) - M S_{m-1}(y)}{M S_m(y) - M^{-1} S_{m-1}(y)} \quad \text{if~}  k=2m+1.
\]
See e.g. \cite{Tr, Pe}.

\section{The case of $C(2m, \pm 2n)$}

In this section we prove the existence of continuous families of  hyperbolic $\mathrm{SL}_2(\BR)$-representations of knot groups of $C(2m, \pm 2n)$ and hence fix the gaps in \cite{HT-genus1, Tr}.

\begin{proposition} \label{solution-even1}
There exist $n-1$ continuous real functions $x_j\colon (2, \infty) \to (0, \infty)$, where $1 \le j \le n-1$, in the variable $y$ such that $R_{C(2m,2n)}(x_j(y),y)=0$ and 
\[
y+2 + \frac{4\sin^2 \frac{(2j-1)\pi}{4n+2}}{(y-2) S^2_{m-1}(y)} < x_j(y) < y+2 + \frac{4\sin^2 \frac{(2j+1)\pi}{4n+2}}{(y-2) S^2_{m-1}(y)} 
\]
for all $y > 2$. 
\end{proposition}

\begin{proof}
Let $K=C(2m,2n)$. We have $R_K(x,y)=S_{-n}(t) - z S_{-n-1}(t)$
where 
\begin{eqnarray*}
t &=& 2+ (y+2-x)(y-2) S^2_{m-1}(y), \\
z &=& 1 + (y+2-x) S_{m-1}(y) (S_{m}(y) - S_{m-1}(y)).
\end{eqnarray*}
Note that $R_K(x,y)= (t - z) S_{-n-1}(t) - S_{-n-2}(t) = S_{n}(t) - (t-z) S_{n-1}(t)$. 

Let $t_j=2\cos \frac{(2j-1)\pi}{2n+1}$ for $j = 1, \dots, n$. 
Then, 
Lemma \ref{chev}(3) gives  $S_n(t)-S_{n-1}(t) = \prod_{j=1}^n (t-t_j)$,
and the signs of $S_n(t_j)$ change alternately as  $S_n(t_1)>0, S_n(t_2)<0, \dots$, 
because of the inequality 
\[
\frac{j-1}{n+1}<\frac{2j-1}{2n+1}<\frac{j}{n+1}\ (j=1,2,\dots,n)
\]
and Lemma \ref{chev}(2).


Fix a real number $y > 2$.  Let $s_j(y) = y+2 + \frac{2-t_j}{(y-2) S^2_{m-1}(y)}$ for $j = 1, \dots, n$. Since $-2 < t_{n} < \dots < t_1 < 2$, we have $s_{n}(y) > \dots > s_1(y) > y+2$. At $x=s_{j}(y)$ we have $t=t_{j}$ and so $S_n(t)=S_{n-1}(t)$. This implies that 
\begin{eqnarray*}
R_K(s_{j}(y),y) &=& (1-(t-z))S_n(t_{j}) \\
&=&(y+2-s_{j}(y))S_{m-1}(y) ( S_{m-1}(y) - S_{m-2}(y) )S_n(t_{j}) \\
&=&- \frac{2-t_j}{(y-2) S_{m-1}(y)}  \, (S_{m-1}(y) - S_{m-2}(y) )S_n(t_{j}).
\end{eqnarray*}
Since $y > 2$, we have $S_{m-1}(y) - S_{m-2}(y) >0$ and $S_{m-1}(y) > 0$ by Lemma \ref{chev}. 
Hence $R_K(s_j(y),y)$ and $S_n(t_j)$ have opposite signs, so
the sign of $R_K(s_j(y),y)$ changes alternately as
$R_K(s_1(y),y)<0, R_K(s_2(y),y)>0,\dots$.

For each $1 \le j \le n-1$, since $R_K(s_{j}(y),y) R_K(s_{j+1}(y),y)<0$,
there exists $x_j(y) \in (s_{j}(y), s_{j+1}(y))$ such that $R_K(x_j(y), y)=0.$ 
Also, since $R_K(y+2,y)=1$ and $R_K(s_1(y), y) < 0$, there exists $x_0(y) \in (y+2, s_1(y))$ such that $R_K(x_0(y), y)=0$. 

Since $R_K(x,y) = S_{n}(t) - (t-z) S_{n-1}(t) = z S_{n-1}(t) - S_{n-2}(t)$, 
we see that $R_K(x,y)$ is a polynomial of degree $n$ in $x$ for each fixed real number $y > 2$. This polyomial has exactly $n$  simple real roots $x_0(y), \dots, x_{n-1}(y)$ satisfying $x_{n-1}(y) >  \dots > x_0(y) > y+2$, hence the implicit function theorem implies that each $x_j(y)$ is a continuous function in $y > 2$.  The continuous functions $x_1(y), \dots, x_{n-1}(y)$ satisfy the conditions of Proposition \ref{solution-even1}. 
\end{proof}

\begin{proposition} \label{solution-even2}
There exist $n-1$ continuous real functions $x_j\colon (2, \infty) \to (0, \infty)$, where $1 \le j \le n-1$, in the variable $y$ such that $R_{C(2m,-2n)}(x_j(y),y)=0$ and 
\[
y+2 + \frac{4\sin^2 \frac{(2j-1)\pi}{4n+2}}{(y-2) S^2_{m-1}(y)} < x_j(y) < y+2 + \frac{4\sin^2 \frac{(2j+1)\pi}{4n+2}}{(y-2) S^2_{m-1}(y)} 
\] 
for all $y > 2$. 
\end{proposition}

\begin{proof}
Let $K=C(2m,-2n)$. We have $R_K(x,y)=S_{n}(t) - z S_{n-1}(t)$
where 
\begin{eqnarray*}
t &=& 2+ (y+2-x)(y-2) S^2_{m-1}(y), \\
z &=& 1 + (y+2-x) S_{m-1}(y) (S_{m}(y) - S_{m-1}(y)).
\end{eqnarray*}

Fix a real number $y  \ge 2$. Choose $t_j$ and $s_j(y)$ for $1 \le j \le n$ as in the proof of Proposition \ref{solution-even1}. 
Recall $S_n(t)=S_{n-1}(t)$ at $t=t_j$.
Since 
\begin{eqnarray*}
R_K(s_{j}(y),y) &=& (1-z)S_n(t_{j}) \\
&=& -(y+2-s_{j}(y))S_{m-1}(y) ( S_{m}(y) - S_{m-1}(y) )S_n(t_{j}) \\
&=& \frac{2-t_j}{(y-2) S_{m-1}(y)}  \, (S_{m}(y) - S_{m-1}(y) )S_n(t_{j}),
\end{eqnarray*}
$R_K(s_j(y),y)$ and $S_n(t_j)$ have the same sign.
So, the sign of $R_K(s_j(y),y)$ changes alternately.
 Since $R_K(s_{j}(y),y) R_K(s_{j+1}(y),y)<0$, there exists $x_j(y) \in (s_{j}(y), s_{j+1}(y))$ such that $R_K(x_j(y), y)=0$ for each $1 \le j \le n-1$. 

We now claim that there exists $x_0(y) \in (0, y+2)$ such that $R_K(x_0(y), y)=0$. Indeed, at $x=0$ we have $t=2+(y^2-4)S^2_{m-1}(y)$ and $z= 1 + (y+2) S_{m-1}(y) (S_{m}(y) - S_{m-1}(y))$. Write $y = \xi + \xi^{-1}$ for some $\xi>1$. 
Then $S_k(y) = \frac{\xi^{k+1} - \xi^{-(k+1)}}{\xi - \xi^{-1}}$ for all integers $k$.

\begin{claim}\label{cl:even2}
\begin{enumerate}
\item
$t = \xi^{2m} + \xi^{-2m}$.
\item
$z = \xi^{-2m} \, \frac{\xi^{4m+1} -1}{\xi-1}$.
\end{enumerate}
\end{claim}

\begin{proof}[Proof of Claim \ref{cl:even2}]
(1) follows from
\begin{align*}
t&=2+(y^2-4)S_{m-1}^2(y)\\
&=2+(\xi-\xi^{-1})^2 \left( \frac{\xi^m-\xi^{-m}}{\xi-\xi^{-1}} \right)^2\\
&=2+(\xi^m-\xi^{-m})^2\\
&=\xi^{2m}+\xi^{-2m}.
\end{align*}

(2) Similarly,
\begin{align*}
z&=1+(\xi+\xi^{-1}+2)\frac{\xi^m-\xi^{-m}}{\xi-\xi^{-1}}
\left(   
\frac{\xi^{m+1}-\xi^{-(m+1)}}{\xi-\xi^{-1}}-\frac{\xi^m-\xi^{-m}}{\xi-\xi^{-1}}
\right)\\
&=\frac{1}{(\xi-\xi^{-1})^2}
\bigl(
1+(\xi+\xi^{-1}+2)(\xi^m-\xi^{-m})
(\xi^{m+1}-\xi^{-(m+1)}-\xi^m+\xi^{-m})
\bigr)\\
&=\frac{1}{(\xi-\xi^{-1})^2}\left(
\xi^{2m+2}+\xi^{-(2m+2)}+\xi^{2m+1}+\xi^{-(2m+1)}-\xi^{2m}-\xi^{-2m}-\xi^{2m-1}-\xi^{-(2m-1)}
\right)\\
&=\frac{\xi^2-1}{(\xi-\xi^{-1})^2}
(
\xi^{2m}+\xi^{2m-1}-\xi^{-(2m+1)}-\xi^{-(2m+2)}
)\\
&=\frac{\xi^2}{\xi^2-1} \xi^{-(2m+2)}
(
\xi^{4m+2}+\xi^{4m+1}-\xi-1
)\\
&=\xi^{-2m}\frac{\xi^{4m+1}-1}{\xi-1}.
\end{align*}
\end{proof}

By Claim \ref{cl:even2}(1), $S_k(t) = \frac{\xi^{2m(k+1)} - \xi^{-2m(k+1)}}{\xi^{2m} - \xi^{-2m}}$ for all integers $k$. 
Hence we have
\begin{align*}
R_K(0,y) &= S_n(t)-zS_{n-1}(t)\\
&= \frac{\xi^{2m(n+1)} - \xi^{-2m(n+1)}}{\xi^{2m} - \xi^{-2m}} - \xi^{-2m} \frac{\xi^{4m+1} -1}{\xi-1}  \frac{\xi^{2mn} - \xi^{-2mn}}{\xi^{2m} - \xi^{-2m}}  \\
&=\frac{1}{\xi^{2m}-\xi^{-2m}}
\left(
(\xi^{2m(n+1)}-\xi^{-2m(n+1)})-\xi^{-2m}\frac{\xi^{4m+1}-1}{\xi-1}(\xi^{2mn}-\xi^{-2mn})
\right)\\
&=
\frac{\xi^{-2mn}}{\xi^{2m}-\xi^{-2m}}
\left(
(\xi^{2m(2n+1)}-\xi^{-2m})-\xi^{-2m} \frac{\xi^{4m+1}-1}{\xi-1}(\xi^{4mn}-1)
\right)
\\
&=\frac{\xi^{-2mn}}{\xi^{2m}-\xi^{-2m}}
\frac{1}{\xi-1}
(
\xi^{4mn+2m}-\xi^{4mn-2m}-\xi^{2m+1}+\xi^{-(2m-1)}
)
\\
&= - \xi^{-2mn} \, \frac{\xi^{4mn} -\xi}{\xi-1} <0.
\end{align*}

Since $R_K(y+2,y)=1$, there exists $x_0(y) \in (0, y+2)$ such that $R_K(x_0(y), y)=0$.  

By writing $R_K(x,y) = S_{n}(t) - z S_{n-1}(t) = (t-z) S_{n-1}(t)- S_{n-2}(t)$ and noting that
\[
t- z = 1+ (y+2-x) S_{m-1}(y) (-S_{m-1}(y) + S_{m-2}(y)),
\]
we see that $R_K(x,y)$ is a polynomial of degree $n$ in $x$ for a fixed real number $y > 2$. This polynomial has exactly $n$ simple real roots $x_0(y), \dots, x_{n-1}(y)$ satisfying $x_{n-1}(y) >  \dots > x_1(y) > y+2 > x_0(y) > 0$, hence the implicit function theorem implies that each $x_j(y)$ is a continuous function in $y > 2$. The continuous functions $x_1(y), \dots, x_{n-1}(y)$ satisfy the conditions of Proposition \ref{solution-even2}.
\end{proof}

\section{The case of $C(2m+1, -2n)$}

In this section we prove the existence of continuous families of hyperbolic $\mathrm{SL}_2(\BR)$-representations of the knot groups of $C(2m+1, -2n)$ and use it to give a proof of Theorem \ref{main}.

\subsection{Real roots of the Riley polynomial}
\begin{proposition} \label{solution-odd}
There exists a unique continuous real function $x\colon (2 \cos  \frac{\pi}{2m+1}, \infty) \to (0, \infty)$ in the variable $y$ such that $R_{C(2m+1,-2n)}(x(y),y)=0$ and $x(y) > y+2$ for all $y > 2 \cos  \frac{\pi}{2m+1}$. 
\end{proposition}

\begin{proof}
Let $K=C(2m+1,-2n)$. We have $R_{K}(x,y)= S_n(t) -  z S_{n-1}(t)$ where 
\begin{eqnarray*}
t &=& 2- (y+2-x)(S_m(y)-S_{m-1}(y))^2,\\
z &=& 1 - (y+2-x)S_m(y) ( S_m(y)-S_{m-1}(y) ).
\end{eqnarray*}

Choose $t_j$ for $1 \le j \le n$ as in the proof of Proposition \ref{solution-even1}. Fix a real number $y > 2 \cos  \frac{\pi}{2m+1}$. Then $S_m(y) > S_{m-1}(y) >0$
by Lemma \ref{chev}(3).
Let $s_j(y) = y+2 - \frac{2-t_j}{(S_m(y)-S_{m-1}(y))^2}$ for $j = 1, \dots, n$. Then $s_{n}(y) < \dots < s_1(y)< y+2$. Since 
\[
R_K(s_{j}(y),y)=(1-z)S_n(t_{j})=(y+2-s_{j}(y))S_m(y) ( S_{m}(y) - S_{m-1}(y) )S_n(t_{j}), 
\]
$R_K(s_j(y),y)$ and $S_n(t_j)$ have the same sign.
Thus the sign of $R_K(s_j(y),y)$ changes alternately.
Hence, there exists $x_j(y) \in (s_{j+1}(y), s_{j}(y))$ such that $R_K(x_j(y), y)=0$ for each $1 \le j \le n-1$. 

By writing $R_K(x,y) = (t-z) S_{n-1}(t) - S_{n-2}(t)$ and noting that
\begin{eqnarray*}
t-z = 1+ (y+2-x)( S_{m}(y) - S_{m-1}(y)) S_{m-1}(y),
\end{eqnarray*}
we see that $R_K(x,y)$ is a polynomial of degree $n$ in $x$ with negative highest coefficient for each fixed real number $y > 2 \cos  \frac{\pi}{2m+1}$. 
For, $t$ (resp. $t-z$) has degree one in the variable $x$ with positive (resp. negative) coefficient,
and the Chebyshev polynomials $S_{n-1}(t)$ and $S_{n-2}(t)$
are polynomials of degree $n-1$ and $n-2$, respectively, in $t$ with positive highest coefficient.
Since $\lim_{x \to \infty} R_K(x,y) = -\infty$ and $R_K(y+2,y)=1$, there exists $x_0(y) \in (y+2, \infty)$ such that $R_K(x_0(y),y)=0$. For a fixed real number $y > 2 \cos  \frac{\pi}{2m+1}$, the polynomial $R_K(x,y)$ of degree $n$ in $x$ has exactly $n$  simple real roots $x_0(y), \dots, x_{n-1}(y)$ satisfying $x_{n-1}(y) <  \dots <  x_1(y) < y+2 < x_0(y) $, hence the implicit function theorem implies that each $x_j(y)$ is a continuous function in $y > 2 \cos  \frac{\pi}{2m+1}$.  

Letting $x(y) = x_{0}(y)$ for $y > 2 \cos  \frac{\pi}{2m+1}$, we have $x(y) > y+2$ and $R_K(x(y),y)=0$. 
\end{proof}

\begin{proposition} \label{solution-odd-limit}
The continuous real function $x(y)$ in Proposition \ref{solution-odd} satisfies the following properties:
\begin{enumerate}
\item $x(y) > 2 + \frac{S_m(y)}{S_{m-1}(y)} +  \frac{S_{m-1}(y)}{S_{m}(y)}> 4$ for all $y > 2 \cos  \frac{\pi}{2m+1}$,
\item $x(y) \to \infty$ as $y \to (2 \cos  \frac{\pi}{2m+1})^+$, 
\item  $y^{2m+2n-2} \big( x(y) - 2 - \frac{S_m(y)}{S_{m-1}(y)} - \frac{S_{m-1}(y)}{S_{m}(y)} \big) \to 1$ as $y \to \infty$.
\end{enumerate}
\end{proposition}

\begin{proof}
(1)
Since $R_K(x(y),y)=0$ we have $S_n(t) = z S_{n-1}(t)$.
By Lemma \ref{chev}(1), $S^2_n(t) - t S_n(t)S_{n-1}(t) + S^2_{n-1}(t) =1$.
Thus we have $(z^2 - t z +1)S^2_{n-1}(t) =1$. Let $G = S_m(y)$ and $H=S_{m-1}(y)$. 
Then $G>H>0$ for $y > 2 \cos  \frac{\pi}{2m+1}$ by Lemma \ref{chev}(3). 
By using $G^2-yGH+H^2=1$ and
 $t - 2 =(x-y-2) (G-H)^2$,
we have
\begin{align*}
z^2 - t z+1 &= (z-1)^2 - (t-2)z \\
&=(x-y-2)^2 G^2(G-H)^2-(x-y-2)(G-H)^2\left( 1+(x-y-2)G(G-H)  \right)
\\
&=(x-y-2)(G-H)^2( (x-y-2)GH-1)\\
&= (x-y-2)(G-H)^2 ((x-2)GH - G^2-H^2) \\
&= (t-2) ((x-2)GH - G^2-H^2). 
\end{align*}
Hence $(t-2) ((x-2)GH - G^2-H^2)S^2_{n-1}(t) =1$. Since $t - 2 =(x-y-2) (G-H)^2 > 0$, we get $(x-2)GH - G^2-H^2 > 0$. This implies that 
$$x > 2 + \frac{G^2+H^2}{GH} = 4 + \frac{(G-H)^2}{GH} > 4.$$

(2)
As $y \to (2 \cos  \frac{\pi}{2m+1})^+$ we have $G-H \to 0$ by Lemma \ref{chev}(3).
 If $x$ is bounded, then $t - 2 =(x-2-y) (G-H)^2 \to 0$ and $1=(t-2) ((x-2)GH - G^2-H^2)S^2_{n-1}(t) \to 0$, a contradiction. Hence $x(y) \to \infty$ as $y \to (2 \cos  \frac{\pi}{2m+1})^+$. 

(3)
We now consider the case $y>2$. 
First,
\begin{align*}
t+2-x&= 
(x-y-2)(G-H)^2+4-x\\
&=(x-y-2)(1+(y-2)GH)+4-x\\
&=(x-y-2)(y-2)GH-y+2\\
&=(y-2) \left( (x-y-2)GH-1\right)\\
&=(y-2)((x-2)GH-G^2-H^2)>0.
\end{align*}
Hence $t>x-2>y$. Noting that $(t-2) ((x-2)GH - G^2-H^2)S^2_{n-1}(t) =1$, we have $(x-2)GH - G^2-H^2 \to 0$ as $y \to \infty$. This is equivalent to $GH(x-2-y) - 1 = GH \frac{t-2}{(G-H)^2} -1 \to 0$ as $y \to \infty$. 
Since
\[
GH \frac{t-2}{(G-H)^2}=\frac{\frac{t}{y}-\frac{2}{y}}{1-\frac{2}{y}+\frac{1}{yGH}},
\]
$\frac{t}{y} \to 1$ as $y \to \infty$. 
The equation $(t-2)  ((x-2)GH - G^2-H^2) S^2_{n-1}(t) =1$ gives
\[
(t-2)GH \left(x-2-\frac{G}{H}-\frac{H}{G}\right)S_{n-1}^2(t)=1.
\]
Since $G$ and $H$ have degree $m$ and $m-1$ respectively in $y$ with positive highest coefficient
and $S_{n-1}(t)^2$ has degree $2n-2$ in $t$, 
 $y^{2m+2n-2}(x-2-\frac{G}{H}-\frac{H}{G})\to 1$ as $y \to \infty$. This completes the proof of Proposition \ref{solution-odd-limit}.
\end{proof}

\subsection{Proof of Theorem \ref{main}}

Let $X$ be the complement of an open tubular neighborhood of $K=C(2m+1,-2n)$ in $S^3$, and $X_r$ the 3-manifold obtained from $S^3$ by $r$-surgery along $K$. Consider the function $x(y)$ in Proposition \ref{solution-odd}. For each $y > 2 \cos  \frac{\pi}{2m+1}$, we have $x(y) > 2 + \frac{S_m(y)}{S_{m-1}(y)} +  \frac{S_{m-1}(y)}{S_{m}(y)}> 4$ by Proposition \ref{solution-odd-limit}(1). Let $M(y)=\frac{1}{2} (\sqrt{x(y)} + \sqrt{x(y)-4}) > 1$, then $\sqrt{x(y)}=M(y)+M(y)^{-1}$. Since $R_K(x(y), y)=0$, there exists a non-abelian representation $\rho_y\colon \pi_1(X) \to \mathrm{SL}_2(\BR)$ such that $$\rho_y(a) = \left[ \begin{array}{cc}
M(y)& 1 \\
0 & M(y)^{-1} \end{array} \right] \quad \text{and} \quad 
\rho_y(b) = \left[ \begin{array}{cc}
M(y) & 0 \\
2 - y(x) & M(y)^{-1} \end{array} \right].$$
Moreover, the image of the canonical longitude $\lambda$ corresponding to the meridian $\mu=a$ has the form $\rho_y(\lambda) = \left[ \begin{array}{cc}
L(y) & * \\
0 & L(y)^{-1} \end{array} \right]$, where 
\[
L (y)= - M(y)^{4n} \frac{M(y)^{-1}S_m(y) - M(y) S_{m-1}(y)}{M(y) S_m(y) - M(y)^{-1} S_{m-1}(y)}.
\]

As in the proof of Proposition \ref{solution-odd-limit},
 we let $G = S_m(y)$ and $H=S_{m-1}(y)$. Then $G>H>0$ for $y > 2 \cos  \frac{\pi}{2m+1}$ and 
$L(y) = M(y)^{4n} \frac{M(y)^2 - \frac{G}{H}}{M(y)^2 \frac{G}{H}- 1}$. Since $M(y)^2 + M(y)^{-2} = x(y)-2 > \frac{G}{H} + \frac{H}{G}$, we have $M(y)^2 > \frac{G}{H} > 1$. This implies that $$L(y) =  M(y)^{4n} \frac{M(y)^2 - \frac{G}{H}}{M(y)^2 \frac{G}{H}- 1}>0.$$

As $y \to (2 \cos  \frac{\pi}{2m+1})^+$ we have $\frac{G}{H} = \frac{S_m(y)}{S_{m-1}(y)}\to 1$ and $M(y) \to \infty$ (by Proposition \ref{solution-odd-limit}(2)), so $\frac{M(y)^2 - \frac{G}{H}}{M(y)^2\frac{G}{H}- 1} \to 1$. Hence
$$
\frac{\log L(y)}{\log M(y)} = 4n +\frac{\log \frac{M(y)^2 - \frac{G}{H}}{M(y)^2\frac{G}{H}- 1}}{\log M(y)}\to 4n.
$$

As $y \to \infty$ we have $y^{2m+2n-2} \big( x(y)- 2 - \frac{G}{H} - \frac{H}{G} \big) \to 1$
by Proposition \ref{solution-odd-limit}(3).
 This is equivalent to $y^{2m+2n-2} (M(y)^2 - \frac{G}{H}) (1-\frac{1}{M(y)^2 \frac{G}{H}}) \to 1$, which implies that $y^{2m+2n-2} (M(y)^2 - \frac{G}{H}) \to 1$. 
 Thus $M(y)^2-\frac{G}{H}\to 0$.
 Since $G$ and $H$ have degree $m$ and $m-1$ respectively in $y$ with positive highest coefficient, $\frac{G}{H}-y\to 0$.
 Then $M(y)^2-y\to 0$.
 Asymptotically, $M(y)^2-\frac{G}{H}\sim y^{2-2m-2n}$ and $M(y)^2 \frac{G}{H}-1\sim y^2$.  
 Hence 
$$
\frac{\log L(y)}{\log M(y)} = 4n +\frac{\log \frac{M(y)^2 - \frac{G}{H}}{M(y)^2\frac{G}{H}- 1}}{\log M(y)} \to 4n - (4m+4n) = -4m.
$$ 

Consider the continuous function $f(y) := -\frac{\log L(y)}{\log M(y)}$ for $y > 2 \cos  \frac{\pi}{2m+1}$. Then from the above arguments we conclude that the image of $f$ contains the interval $(-4n, 4m)$. This implies that for any slope $r \in (-4n, 4m)$ there exists $y > 2 \cos  \frac{\pi}{2m+1}$ such that $r  = f(y) = -\frac{\log L(y)}{\log M(y)}$. 

The continuous family $C$ of nonabelian representations $\{\rho_y\}$, $y > 2 \cos  \frac{\pi}{2m+1}$, contains a special one which is the reducible nonabelian representation $\rho_2$ (i.e. $\rho_y$ at $y=2$). This representation has index $0$ by Corollary \ref{c1} and so by Lemma \ref{l1} the continuous family $C$ has index $0$. Applying Proposition \ref{p1}(1) to $C$, with $a(y) = \log M(y)$ and $b(y) = \log L(y)$, we conclude that any slope $r \in (-4n, 4m)$ is a left orderable slope of $K=C(2m+1, -2n)$ detected by hyperbolic $\mathrm{SL}_2(\BR)$-representations of the knot group. 


\end{document}